\documentclass[12pt, reqno]{amsart}
\usepackage{amsmath,times,epsfig,amssymb,amsbsy,amscd,amsfonts,amstext,color,bm,mathrsfs}
\usepackage[arrow,matrix]{xy}
\usepackage{hyperref}

\usepackage{bm}

\usepackage{tikz}

\numberwithin{equation}{section}

\theoremstyle{plain}
\newtheorem{theorem}{Theorem}[section]
\newtheorem{corollary}[theorem]{Corollary}
\newtheorem{lemma}[theorem]{Lemma}

\newtheorem{conjecture}[theorem]{Conjecture}

\theoremstyle{definition}
\newtheorem{definition}[theorem]{Definition}
\newtheorem{example}[theorem]{Example}

\theoremstyle{remark}
\newtheorem{remark}[theorem]{Remark}

\newcommand{\A}{\mathcal{A}}

\newcommand{\D}{\mathcal{D}}

\begin{document}

\title{The separator-based derivations for arrangements of monomial groups}

\begin{abstract}
We investigate derivation modules of hyperplane arrangements associated with the monomial groups $G(r,p,\ell)$. Motivated by M\"uhlherr's separator-based construction for graphic arrangements, we extend this approach by combining separators with integral expressions. For a subarrangement $\A_G$ defined by a graph $G$, we show that, when $p<r$, the derivation module $\D(\A_G)$ is generated by the standard derivations $\eta_0,\dots,\eta_{\kappa(G)}$ together with a finite set of separator-based derivations. This result generalizes the complete graph case and provides new evidence toward the broader applicability of separator methods in the study of hyperplane arrangements.
\end{abstract}

\author{Zixuan Wang}
\address{Zixuan Wang, 
Center for University Mathematics Teaching,
College of Science,
Nanjing Forestry University,
Longpan Road 159,
Nanjing, 210037, China
}
\email{zixuan.wang@njfu.edu.cn}

\author{Weili Guo}
\address{Weili Guo, 
Department of Mathematics, College of Mathematics and Physics, Beijing University of Chemical Technology,
North Third Ring East Road 15,
Beijing, 100029, China
}
\email{guowl@mail.buct.edu.cn }


\subjclass[2020]{Primary 52C35, Secondary 20F55, 13N15}
\keywords{Hyperplane arrangements, Monomial groups, Separator-based derivations}


\maketitle


\section{Introduction}
\label{sec:intro}

Let $V$ be an $\ell$-dimensional vector space over a field $\mathbb{K}$.
Denote by $S$ a polynomial ring in variables $x_1, x_2, \dots, x_\ell$.
Let the hyperplane $H \subseteq V$ be an affine subspace of dimension $(\ell-1)$,
which can be expressed as the kernel of a polynomial $\alpha_H$ of degree 1.
Let the hyperplane arrangement $\A$ be a finite set of hyperplanes.
The product $Q(\A)=\prod_{H\in \A} \alpha_H$ is called a defining polynomial of $\A$.
If $Q(\A)$ is a homogeneous polynomial, we say $\A$ is central. 

Let $\partial_i = \frac{\partial}{\partial_{x_i}}$,
and 
define $\mathrm{Der}(S)=\bigoplus_{i=1}^{\ell} S \partial_i$ to be an $S$-module.
The derivation module $\D(\A)$ of a central arrangement $\A$ is defined by
\begin{equation*}
	\D(\A)=\{ \theta \in \mathrm{Der}(S) \mid \theta(\alpha_H) \in \alpha_H S, \text{for \ any \ } H \in \A \}.
\end{equation*}
If $\D(\A)$ is a free $S$-module, then $\A$ is free.
Research on freeness of hyperplane arrangements originates from Saito \cite{saito},
who proved that Coxeter arrangements are free.
It is well known that freeness is a rare property,
and considerable effort has been devoted to classifying arrangements in terms of freeness.
Constructing a minimal set of generators for $\D(\A)$ 
and then 
verifying whether it is a basis by Saito's criterion (Theorem 4.19, \cite{ot})
is
a research method to determine whether $\A$ is free.

A particularly well-studied class of arrangements arises from graphs.
Given a simple graph $G=(V, E)$
with the vertex set $V=\{1,2, \dots, \ell \}$ and the edge set $E \subseteq \bigl\{ \{i , j \} \mid i, j \in V \bigr\}$, 
the graphic arrangement is defined by
$\prod_{ \{i,j\} \in E} (x_i-x_j) $, which is a subarrangement of the Coxeter arrangement of type $A$.
Stanley, Edelman, and Reiner \cite{er, stanley} classified free graphic arrangements:
a graphic arrangement is free if and only the graph $G$ is a chordal graph.
Later, Tran and Tsujie \cite{tt-mat} gave a precise characterization of the so-called MAT-freeness of graphic arrangements,
where MAT-freeness was introduced by Cuntz and M\"ucksch \cite{cm}
as a generalization of the Multiple Addition Theorem of Abe and Terao \cite{at}.
They proved that a graphic arrangement is MAT-free if and only if the graph $G$ is strongly chordal.

More recently, 
attention has turned to constructing generating sets for the derivation modules of graphic arrangements.
Abe, K\"uhne, M\"ucksch, and M\"uhlherr \cite{akmm} 
constructed a generating set for graphic arrangement 
related to the complement of the cycle graph $C_\ell$.
Furthermore, M\"uhlherr \cite{l} 
used the separator-based vector fields 
to obtain a finite generating set for general graphic arrangements,
where a subset $T$ of a vertex set $V$ of a connected graph $G$ is called a separator 
if the induced subgraph $G_{V \backslash T}$ is not connected.
An open question posed in \cite{l} is to 
construct more generators for larger classes of arrangements by this method.
In fact, 
it is difficult to construct a set of generators for $\D(\A)$ of a given arrangement $\A$, 
even when $\A$ is free.

A reflection arrangement is a finite set of reflecting hyperplanes 
that are the fixed subspaces of reflections in a reflection group. 
For more details, see Chapter 6 of \cite{ot}.
We now focus on arrangements associated with monomial reflection groups.
Following the standard description in \cite{FWY} and Chapter 6 in \cite{ot}, let
$V=\bigoplus_{i=1}^{\ell}\mathbb Ce_i$, let $r\geq2$, and let
$C_r$ be a cyclic group generated by $\zeta$, where $\zeta$ is a primitive $r$-th root of unity.
For a permutation $\sigma\in\mathfrak S_\ell$ and a map
$\varepsilon:[\ell]\to C_r$, define
$
\varphi_{\sigma,\varepsilon}(e_i)
=\varepsilon(i)e_{\sigma(i)}.
$
The collection of all such transformations is the full monomial group
$G(r,1,\ell)$. Its order is $r^\ell\ell!$.

Let $p$ be a divisor of $r$. The monomial group $G(r,p,\ell)$ is the
subgroup consisting of the transformations $\varphi_{\sigma,\varepsilon}$
such that
$
\left(\prod_{i=1}^{\ell}\varepsilon(i)\right)^{r/p}=1,$
or equivalently
$\prod_{i=1}^{\ell}\varepsilon(i)
\in\langle\zeta_r^p\rangle=C_{r/p}.
$
Thus $|G(r,p,\ell)|=r^\ell\ell!/p$. The assumption $p\mid r$ ensures
that $r/p$ is an integer and that $\langle\zeta_r^p\rangle$ is the
subgroup of $C_r$ of order $r/p$. 

Indeed, the reflections in $G(r,p,\ell)$ that are relevant here have two
forms. First, for every $p\mid r$, there are reflections of order two
which interchange $e_i$ and $e_j$ up to factors in $C_r$. Their
reflecting hyperplanes are
$
H_{ij}(\zeta)=\ker(x_i-\zeta x_j),
$
for
$
\ \zeta\in C_r,
$
and their defining forms contribute
$
\prod_{\zeta\in C_r}(x_i-\zeta x_j)=x_i^r-x_j^r.
$
Second, a diagonal transformation which multiplies one basis vector by
$\lambda\in C_{r/p}\setminus\{1\}$ and fixes the other basis vectors is
a reflection with reflecting hyperplane $H_i=\ker(x_i)$. Nontrivial
reflections of this second type exist precisely when $p<r$. 
If $p=r$, then $C_{r/p}=C_1=\{1\}$, so no
coordinate hyperplane occurs.

A simple reflection arrangement records the distinct
reflecting hyperplanes, but not the orders of the corresponding
reflections. Hence all groups $G(r,p,\ell)$ with $p<r$ have the same
reflection arrangement, although their group orders are different.  
The case
$p=r$ is different because the coordinate hyperplanes themselves are
absent.

It follows that 
the reflection arrangement $\A$ corresponding to monomial group 
$G(r, p, \ell)$
has defining polynomial
\begin{equation}
	Q(\A)=
	\begin{cases}
		x_1x_2\dots x_\ell\cdot\prod\limits_{1\leq i<j\leq\ell}(x_i^r-x_j^r), 
		&\mbox{ if }r>p, \\
		\prod\limits_{1\leq i<j\leq\ell}(x_i^r-x_j^r), &\mbox{ if }r=p. 
	\end{cases}
\end{equation}
When $r=2$, the above arrangements are of type $B$ or type $D$.
Hoge, Mano, R\"ohrle, and Stump \cite{hmrs} gave some free multiplicities on the above arrangements.
By integral expressions, the bases were given by Feigin, Wang, and Yoshinaga \cite{FWY}. 
In fact, freeness of subarrangements of $\A$ corresponding to $G(2,p,\ell)$ is widely studied by many researchers.
These classes of subarrangements are called signed graphic arrangements which are related to signed graphs.

\begin{definition}
	\label{def:AG}
	Given a simple graph $G$,
	define the subarrangement $\A_G$ of arrangements of type $G(r, p, \ell)$ by 
	\begin{equation}
		Q(\A_G)=
		\begin{cases}
			x_1x_2\dots x_\ell\cdot\prod\limits_{ \{i, j\} \in E(G) }(x_i^r-x_j^r), 
			&\mbox{ if }r>p, \\
			\prod\limits_{ \{i, j\} \in E(G) }(x_i^r-x_j^r), &\mbox{ if }r=p. 
		\end{cases}
	\end{equation}  
\end{definition}

For $r=2$,
Bailey \cite{b-pre} gave the conditions on graph $G$ to obtain some classes of free arrangements $\A_G$.
Torielli and Tsujie \cite{tt} extended Bailey's results to get more free subarrangements of $G(2, p, \ell)$.
The construction of generators for $\D(\A_G)$ 
of type $G(r, p, \ell)$, in both the free and non-free cases for $r\ge 2$, remains largely unexplored.
Since the derivation module $\D(\A_G)$ 
can be written as a direct sum of $\D(\A_C)$ for connected components $C$ of $G$,
we always consider the connected graph $G$ in this paper.

\begin{definition}
	\label{def:separator-based}
	Given a simple connected graph $G$, a separator $T$, a connected component $C$ in $G_{V \backslash T}$, and a polynomial $P\in S[y]$,
	define a vector field $\theta_C^{T, P}$ by
	\begin{equation}
		\theta_C^{T, P}=\sum_{i\in C} \prod_{u\in T} (x_i^r-x_u^r) \int_0^{x_i} P(t^r) dt \partial_i. 
	\end{equation}   
\end{definition}

Let $K_\ell$ be a complete graph. 
Note that $\D(\A_{K_\ell}) \subseteq \D(\A_{G})$ for arbitrary graph $G$,
and the basis of $\D(\A_{K_\ell})$ is well known.

\begin{theorem}
	\label{thm:basis for complete graph} 
	 $($$\mathrm{Proposition\ 6.77}$, \cite{ot}$)$
	Let $\eta_k= \sum_i \int_0^{x_i} t^{kr} dt \partial_i \doteq \sum_i x_i^{kr+1} \partial_i$. 
	The vector fields $\eta_0, \eta_1, \dots, \eta_{\ell-1}$ 
	form a basis for $\D(\A_{K_\ell})$ of type $G(r, p, \ell)$ where $r>p$. 
\end{theorem}

\begin{remark}
	\label{rem:reduction}
	We have $\eta_k\in \langle \eta_0, \eta_1, \dots, \eta_{\ell-1}  \rangle$ for $k \geq \ell$,
	since $\eta_k\in \D(\A_G)$. 
\end{remark}

The main theorem of the paper is as follows.

\begin{theorem}
	\label{thm:main}
	Let $\kappa(G)$ be the connectivity of graph $G$.
	The vector fields $\eta_0, \eta_1, \dots, \eta_{\kappa(G)}$ and a finite set of $\theta_C^{T, 1}$ form a generating set for arrangement $\A_G$ of type $G(r, p, \ell)$ where $p<r$. 
\end{theorem}


In this paper, we fix integers $r, p, \ell$. Although the graph $G$ may vary, every arrangement under consideration is a graph-related subarrangement $\mathcal A_G$ of the reflection arrangement of type $G(r,p,\ell)$, as defined in Definition \ref{def:AG}. 
The paper is organized as follows. 
In \S \ref{subsec:grobner}, we examine Gr\"obner bases for some ideals in $S$.
Then we observe a generating set for the intersection of these ideals.
Some basic concepts of graph theory are given in \S \ref{subsec:action}.
Later the action of $\theta_C^{T, P}$ at some polynomials divide the defining polynomial $Q(\A_G)$
is analyzed.
In \S \ref{sec:proof}, 
the main theorem \ref{thm:main} is proved.
Finally, a conjecture concerning the arrangement $\A_G$ of type $G(r, r, \ell)$ is posed.

\section{The characterization of $\theta_C^{T, P}$}
\subsection{A Gr\"obner basis}\label{subsec:grobner}
A Gr\"obner basis yields many important properties of a generating set of an ideal,
and it is the tool we use below.

Fix an index $v\in\{1, 2, \dots, \ell \}$.
For $1\le j \le m$, define the subset $B_j=\{v_1^j, v_2^j, \dots, v_{n_j}^j\} \subseteq \{1, 2, \dots, \ell \}$, 
and the ideal $I_j=\langle x_v^r-x_i^r \mid i\in B_j,\ v\neq i \rangle$. 
Let
$
I=\bigcap_{j=1}^{m} I_j .
$
Our aim in this subsection is to make an explicit set of generators for the
intersection $I$ (Lemma~\ref{lem:z_new}) and to prove the statement
$I\cap\langle x_v^k\rangle=I\cdot\langle x_v^k\rangle$
(Lemma~\ref{lem:xvf}).
Up to this point no graph is involved, 
so $B_j$ may simply be regarded as a subset of the index set. 
It may be helpful to know in advance the situation in which
Lemma~\ref{lem:z_new} and Lemma~\ref{lem:xvf} will be used: 
$B_j$ will be
the set of internal vertices of a path from $v$ to $w$ in a graph $G$
with $\{v,w\}\notin E(G)$, listed in the order in which they occur along
that path.
Distinct sets $B_j$
correspond to distinct paths from $v$ to $w$.

Introduce the term order $\preceq$ on $S$ such that $x_v \prec x_i$ for $i\ne v$.
By $S$-polynomials and Buchberger's algorithm, we can check 
$F_j=\{ x_v^r-x_i^r \mid i\in B_j,\ v\neq i \}$ is a Gr\"obner basis for the ideal $I_j$. 
Set $f_i=x_v^r-x_i^r$. We have the leading power product $lp(f_i)=x_i^r$, the leading term $lt(f_i)=-x_i^r$,
and $ L_{i,j}=lcm\bigl(lp(f_i), lp(f_j)\bigr)=lcm(x_i^r, x_j^r)=x_i^r x_j^r$.
The $S$-polynomial of $f_i, f_j$ is 
\begin{equation*}
	S(f_i, f_j)=\frac{ L_{i,j}}{lt(f_i)} f_i-\frac{ L_{i,j}}{lt(f_j)} f_j=x_v^r(x_i^r-x_j^r).
\end{equation*}
Since $x_v^r(x_v^r-x_j^r)=S(f_i, f_j)-\frac{x_v^r x_i^r}{lt(f_i)}f_i$ and $0=x_v^r(x_v^r-x_j^r)+\frac{x_v^r x_j^r}{lt(f_j)}f_j$,
the $S$-polynomial $S(f_i, f_j)$ reduces to $0$ modulo $f_i, f_j$.
In conclusion, $F_j$ is a Gr\"obner basis. See \cite{ap} for details.

Let $U \subseteq  \{1, 2, \dots, \ell \}$ be a subset such that $U\cap B_j \ne \emptyset$ for $1\le j \le m$. 
Fix $ v  \in \{1, 2, \dots, \ell \} $, define the set $Z_{ new }$ by
\begin{equation*}
	Z_{ \mathrm{new} }=\bigl\{ \prod_{u\in U} (x_v^r-x_u^r) \mid U \subseteq \{1,2,\dots,\ell\}, \ 
	U \cap B_j \neq \emptyset \ (1 \leq j \leq m) \bigr\}.
\end{equation*}

The proof of the following lemma is similar to that of Lemma 3.1 in \cite{l} and Theorem 2.3 in \cite{sst}.

\begin{lemma}
	\label{lem:z_new}
	Let $I=\cap_{1\le j \le m} I_j$.
	The ideal generated by $Z_{ \mathrm{new} }$ is exactly the ideal $I$.
\end{lemma}

\begin{proof}
	It is clear that $\langle Z_{ \mathrm{new} } \rangle \subseteq I$.
	We prove the converse inclusion.
	Assume $f\in I$. 
	Since the elements of $F_j$ form a Gr\"obner basis for the ideal $I_j$,
	there exists $f_{u_j} \in F_j  $ 
	such that $lt(f_{u_j} ) = -x_{u_j}^r\mid lt(f) $ 
	for $1\le j \le m$.
	Let \(U=\{u_1,\ldots,u_m\}\), where repeated indices are discarded, and set
	$
	g_U=\prod_{u\in U}(x_v^r-x_u^r).
	$
	By construction, \(U\cap B_j\neq\varnothing\) for every \(j\), 
	so
	\(g_U\in Z_{\mathrm{new}}\).
	Since \(x_v\prec x_u\) for
	\(u\neq v\),
	$
	lt(g_U)=(-1)^{|U|}\prod_{u\in U}x_u^r.$
	The indices in \(U\) are distinct, so the monomials \(x_u^r\), \(u\in U\),
	involve distinct variables. Since each of them divides
	$lt(f)$, their product also divides
	$lt(f)$. Therefore,
	$
	lt(g_U)~\mid~lt(f).
	$
	By division, we have $f=h g_U +r$. Then $r\in I$, and $lt(r)\prec lt(f)$.
	We can apply the same process on $r$.
	Consequently, 
	the above reduction process terminates after finitely many
	steps. 
	Its final remainder must be zero.
	By the chosen $g_U\in Z_{ \mathrm{new} }$ at each step, 
	we get $f\in \langle Z_{ \mathrm{new} } \rangle$. Thus $\langle Z_{ \mathrm{new} } \rangle \supseteq I$.
\end{proof}

\begin{remark}
	\label{rem:from above lemma}
	Let $f\in S$, 
	and $k > 0$. 
	If $f\in I \cap \langle x_v^k \rangle$, 
	by an argument similar to the proof of the preceding lemma,
	$f$ is generated by $x_v^k \cdot \prod_{u\in U} (x_v^r-x_u^r)$ for some subsets $U$  
	such that $U\cap B_j \ne \emptyset,$ for $1\le j \le m$. 
\end{remark}

\begin{lemma}
	\label{lem:xvf}
	Let \(f\in S\), and let \(k>0\). If \(x_v^k f\in I\), then
	$
	f\in I=\langle Z_{ \mathrm{new} }\rangle.
	$
	In particular,
	$
	I\cap\langle x_v^k\rangle
	=I\cdot\langle x_v^k\rangle.
	$
\end{lemma}

\begin{proof}
	We first prove that
	$
	(I_j:x_v^k)=I_j
	$
	for every \(1\leq j\leq m\). The inclusion
	$
	I_j\subseteq (I_j:x_v^k)
	$
	is immediate. To prove the converse inclusion, let \(0\neq h\in S\) and
	assume that \(x_v^k h\in I_j\). Since
	$
	F_j
	$
	is a Gr\"obner basis for \(I_j\), 
	there exists \(i \in B_j\), \(i \neq v\),
	such that
	$
	lt(f_{i})=-x_{i}^r
	\mid lt(x_v^k h)
	=x_v^k \ lt(h).
	$
	Then
	$
	-x_{i}^r\mid lt (h).
	$
	By division, we have $h=q f_i +r$. 
	Then either the remainder $r=0$ or $lt(r)\prec lt(h)$.
	Moreover,
	$
	x_v^k r
	=x_v^k h-x_v^k q f_{i}\in I_j.
	$
	If $r \neq 0$, we can repeat this process.
	Consequently, 
	the above reduction process terminates after finitely many
	steps. 
	Its final remainder must be zero.
	Hence
	$
	h=\sum_s q_sf_{i_s}\in I_j.
	$
	This proves
	$
	(I_j:x_v^k)=I_j.
	$
	
	Since \(I=\bigcap_{1\leq j\leq m}I_j\), we have
	$
		(I:x_v^k)
		=\left(\bigcap_{1\leq j\leq m}I_j:x_v^k\right)  
		=\bigcap_{1\leq j\leq m}(I_j~:~x_v^k)              
		=\bigcap_{1\leq j\leq m}I_j
		=I.
	$
	Therefore, if \(x_v^k f\in I\), then \(f\in I\). 
	
	Finally, for the principal ideal \(\langle x_v^k\rangle\), we have
	$
	I\cap\langle x_v^k\rangle
	=x_v^k(I:x_v^k).
	$
	Indeed, if \(f\in I\cap\langle x_v^k\rangle\), then \(f=x_v^k g\) for some
	\(g\in S\), and \(f\in I\) implies \(g\in(I:x_v^k)\); the reverse inclusion
	is immediate. 
	Since \((I:x_v^k)=I\), 
	we obtain
	$
	I\cap\langle x_v^k\rangle
	=x_v^kI
	=I\cdot\langle x_v^k\rangle.
	$
\end{proof}

\subsection{The action of $\theta_C^{T, P}$} \label{subsec:action}
In this section, we first introduce some useful concepts in graph theory based on Chapter 1 in \cite{d}.
Given a simple graph $G=(V, E)$, 
we define the neighborhood, induced subgraphs, paths, separators, and connectivity in turn.

\begin{definition}
Given a simple graph $G$,
define the neighborhood $N_G(v)$ of a vertex $v~\in~V $ by $N_G(v)= \bigl\{w \mid \{v, w\} \in E \bigr\}$. 
Define the neighborhood of $W\subseteq V$ by 
$N_G(W)= \bigl\{w \in V \backslash W \mid \{v, w\} \in E, v \in W \bigr\}.$  
\end{definition} 

\begin{definition}
	For $U\subseteq V$, a subgraph $G_U$ of $G$ induced by $U$ is
	$
	G_U
	=\bigl(U,\bigl\{\{i,j\}\in E\mid i,j\in U\bigr\}\bigr).
	$
	For $T\subseteq V$, write $G\backslash T=G_{V\setminus T}$.
\end{definition} 

\begin{definition}
	A path is a non-empty graph $B=(V_B, E_B)$ of the form
	$V_B=\{v_0, v_1, \dots, v_{n+1}\}$, 
	$E_B=\bigl\{ \{v_s, v_{s+1} \} \mid 0\leq s\leq n \bigr\}$, 
	where $v_s$ are all distinct.
\end{definition} 

We often refer to a path by the natural sequence of its vertices,
writing
$B=(v_0, v_1, v_2, \dots, v_n, v_{n+1})$
and 
calling $B$ a path from $v_0$ to $v_{n+1}$ (as well as between $v_0$ and $v_{n+1}$).

\begin{definition}
	Let $v,w\in V$ be distinct and let
	$T\subseteq V\setminus\{v,w\}$. The set $T$ is a
	$(v,w)$-separator if every path from $v$ to $w$ meets $T$.
\end{definition}

We say that $T\subseteq V$ separates $G$ if $T$ is a $(v, w)$-separator for some vertices $v,w\in V$. 
A $(v,w)$-separator $T$ is minimal if no proper subset of $T$ separates
$v,w$.
Define $\mathscr{T}(G)$ to be the set of all minimal separators of $G$,
and $\mathscr{C}(T)$ to be the set of vertex sets of all connected components of $G_{V\backslash T}$. 

\begin{definition}
	For $k\in\mathbb N$, the graph $G$ is \emph{$k$-connected} if
	$|V|>k$ and $G\backslash X$ is connected for every $X\subseteq V$ with
	$|X|<k$. The \emph{connectivity} $\kappa(G)$ is the greatest integer
	$k$ for which $G$ is $k$-connected. In particular,
	$\kappa(K_\ell)=\ell-1$.
\end{definition}

\begin{definition}
If $G$
is a connected graph, $T$ is a separator of $G$, and
\(D\in\mathscr C(T)\), we call \(D\) a \(T\)-full component if
$
N_G(D)=T.
$
\end{definition} 

It is well known that
the separator \(T\) is a minimal separator of $G=(V, E)$ if and only if \(G\setminus T\)
has at least two \(T\)-full components. 
Indeed, if \(T\) is a minimal
\((v, w)\)-separator, the components containing \(v\) and \(w\) are adjacent
to every vertex of \(T\). Conversely, if \(C_1\) and \(C_2\) are two
\(T\)-full components, then, for every \(t\in T\), there is a
path between $v$ and $w$ through \(t\) in
\(G\setminus(T\setminus\{t\})\). Thus no proper subset of \(T\) separates
$v$ and $w$. 


We now turn our attention to the vector field $\theta_C^{T,P}$ introduced in Definition \ref{def:separator-based}.

\begin{lemma}
	\label{lemma:derivation}
	The vector field $\theta_C^{T, P}$ is a derivation for $\A_G$.
\end{lemma}

\begin{proof}
	The arrangement $\A_G$ may consist of hyperplanes $x_i=0$, 
	and hyperplanes $x_i-\zeta x_j=0$ where $\zeta^r = 1$.
	We have $\theta_C^{T, P}(x_i)\in x_i S$, since $\int_0^{x_i} P(t^r) dt$ is divisible by $x_i$.
	It remains to prove $\theta_C^{T, P}(x_i-\zeta x_j)\in (x_i-\zeta x_j) S$ for every edges $\{i, j\}\in E(G)$.
	
	If $i, j\in T$, it is trivial. 
	
	If $i\in C$, and $j\in T$, we have $\theta_C^{T, P}(x_i-\zeta x_j)\in (x_i-\zeta x_j) S$, since $\prod_{u\in T} (x_i^r-x_u^r)$ is divisible by $(x_i-\zeta x_j)$.
	
	If $i, j\in C$, we have
	\begin{align*}		
			\theta_C^{T, P}(x_i-\zeta x_j)
			&= \prod_{u\in T} (x_i^r-x_u^r) \int_0^{x_i} P(t^r) dt-
			\zeta \prod_{u\in T} (x_j^r-x_u^r) \int_0^{x_j} P(t^r) dt \\
			&= \prod_{u\in T} (x_i^r-x_u^r) \int_0^{x_i} P(t^r) dt-
			\prod_{u\in T} (x_j^r-x_u^r) \int_0^{\zeta x_j} P((\zeta^{-1} t')^r) dt' \\
			&= \prod_{u\in T} (x_i^r-x_u^r) \int_0^{x_i} P(t^r) dt -
			\prod_{u\in T} (x_j^r-x_u^r) \int_0^{\zeta x_j} P((t')^r) dt' \\
			&=
			\prod_{u\in T} (x_i^r-x_u^r) \int_{\zeta x_j}^{x_i} P(t^r) dt  \\
			&\equiv 0 \qquad \qquad \qquad \qquad \qquad \mathrm{mod} (x_i-\zeta x_j)
		\end{align*}
	where $t'=\zeta t$.
\end{proof}

We say that $B_j=\{v_1^j, v_2^j, \dots, v_{n_j}^j\} \subseteq V(G)$ gives a path between $v, w$, 
if the ordering $v_1^j, v_2^j, \dots, v_{n_j}^j $ is fixed, 
and $( v=v_0^j, v_1^j, v_2^j, \dots, v_{n_j}^j, w=v_{n_j+1}^j  )$ is a path.
Proposition 3.2 in \cite{l} relates the derivation of graphic arrangements to the path of given graph.
Similarly, we consider the relation between $\theta \in \D(\A_G)$ 
and the path $B_j$, where $\{v, w\}\notin E(G)$, $n_j\ge 1$.
Since $\theta (x_i^r - x_j^r) \in (x_i^r - x_j^r) S$ for $\{i ,j\} \in E(G)$,
we have
\begin{align}
	\label{equ:xvr-xwr}
	\theta(x_v^r-x_w^r) & = \theta (x_v^r-x_{v_1^j}^r+x_{v_1^j}^r- \dots + x_{v_{n_j}^j}^r -x_w^r) \notag \\
	& = \sum_{i=0, 1, \dots, n_j} (x_{v_i^j}^r-x_{v_{i+1}^j}^r) f_{v_i^j,v_{i+1}^j} \notag \\
	& = \sum_{i=0, 1, \dots, n_j} (x_{v_i^j}^r -x_v^r +x_v^r - x_{v_{i+1}^j}^r) f_{v_i^j,v_{i+1}^j} \notag \\
	& = \sum_{i=1, \dots, n_j+1} (x_v^r-x_{v_i^j}^r) g_{v,v_i^j} \notag \\
	& \equiv \sum_{i=1, \dots, n_j} (x_v^r-x_{v_i^j}^r) g_{v,v_i^j} \qquad \mathrm{mod}(x_v-\zeta x_w)
\end{align}
where $f_{v_i^j,v_{i+1}^j}, g_{v,v_i^j} \in S$, and $\zeta ^r =1$. 
Thus $\overline{\theta(x_v^r-x_w^r)}\in I_j$ in $S/\langle x_v-\zeta x_w \rangle$.
If each $B_j$ $(j=1,2,\dots,m)$ gives a path between $v, w$, 
we have $\overline{\theta(x_v^r-x_w^r)}\in \cap_{1\le j \le m} I_j = I= \langle Z_{ \mathrm{new} } \rangle$ by Lemma \ref{lem:z_new}.

\section{Proof of main theorem}
\label{sec:proof}

We construct an infinite set of derivations of $\D(\A_G)$, 
then observe the relation between that of $\D(\A_G)$ and $\D(\A_{G'})$, 
where $G'= \bigl(V(G), E(G)\backslash \{v,w\}\bigr)$ for $ \{v, w\} \in E(G)$.
Note that $\A_G'=\A_G\backslash \{ x_v^r - x_w^r =0\}= \A_{G'}$.

\begin{definition}
	\label{def:generators}
	Let $\mathscr{G}(G)$ be a set consisting of vector fields 
	$\eta_0, \eta_1, \dots, \eta_{\ell-1}$, and
	$\theta_C^{T,P}$ for all minimal separators $T$ of $G$, $C\in \mathscr{C}(T)$, $P\in S[y]$.
\end{definition}

\begin{lemma}
	\label{lem:generators}
	Let \(G\) be a connected graph, let \(\{v,w\}\in E(G)\), and let
	$
	G'=(V(G),E(G)\setminus\{\{v,w\}\})
	$
	be connected. Let \(\mathscr G(G)\) and \(\mathscr G(G')\) be defined as
	above. Then
	$
	\mathscr G(G)\subseteq\langle\mathscr G(G')\rangle .
	$
\end{lemma}

\begin{proof}
	The derivations \(\eta_0,\eta_1,\ldots,\eta_{\ell-1}\) occur in both
	\(\mathscr G(G)\) and \(\mathscr G(G')\). 
	It is enough to consider
	a minimal separator \(T\) of \(G\), a component
	\(C\in\mathscr C_G(T)\), and a polynomial \(P\in S[y]\), and to prove that
	$
	\theta^{T,P}_{C}\in\langle\mathscr G(G')\rangle.
	$
	
	
	Deleting an edge cannot merge components. Hence every component of
	\(G'\setminus T\) is contained in a unique component of \(G\setminus T\),
	and
	\[
	\theta_{C}^{T, P}
	=
	\sum_{\substack{D\in\mathscr C_{G'}(T)\\D\subseteq C}}
	\theta_{D}^{T, P}.
	\]
	It is enough to show each $
	\theta^{T,P}_{D}\in\langle\mathscr G(G')\rangle.
	$
	
	If \(T\) is a minimal separator of \(G'\), every summand on the right-hand
	side belongs to \(\mathscr G(G')\), and the required conclusion follows.
	
	Suppose now that \(T\) is not a minimal separator of \(G'\). Since \(T\)
	is minimal in \(G\), the graph \(G\setminus T\) has at least two
	\(T\)-full components. 
	Deleting one edge can affect the connectivity or
	the neighborhood of at most one of these components. Therefore at least
	one \(T\)-full component remains in \(G'\setminus T\). Since \(T\) is not
	minimal in \(G'\), there cannot be two such components. Consequently,
	\(G'\setminus T\) has a unique \(T\)-full component, denote it by \(D_0\).
	
	Let \(D\in\mathscr C_{G'}(T)\) with \(D\neq D_0\), and set
	$
	T_D=N_{G'}(D).
	$
	Because \(D\) is not \(T\)-full, \(T_D\subsetneq T\). Moreover, \(D\) is a
	\(T_D\)-full component of \(G'\setminus T_D\). Let \(E_0\) be the component
	of \(G'\setminus T_D\) containing \(D_0\). Since \(D_0\) is \(T\)-full,
	every vertex of \(T_D\) has a neighbor in \(D_0\), and hence
	$
	N_{G'}(E_0)=T_D.
	$
	Thus \(D\) and \(E_0\) are two distinct \(T_D\)-full components of
	\(G'\setminus T_D\). By the characterization above, \(T_D\) is a minimal
	separator of \(G'\), and \(D\in\mathscr C_{G'}(T_D)\).
	
	It remains to compare the derivations associated with \(T\) and \(T_D\).
	Write
	$
	P(y)=\sum_{k=0}^{s}p_k y^k
	$
	and define
	$
	\widehat P(y)=\sum_{k=0}^{s}\frac{p_k}{kr+1}y^k .
	$
	Then
	\[
	\int_0^x P(t^r)\,dt=x\,\widehat P(x^r).
	\]
	For a subset \(R\subseteq V(G)\), define
	$
	B_{R,P}(y)
	=\prod_{u\in R}(y-x_u^r)\widehat P(y)
	$
	and
	$
	Q_{R,P}(y)=B_{R,P}(y)+ry\frac{d}{dy}B_{R,P}(y).
	$
	Both polynomials belong to \(S[y]\). 
	We have
	\[
	\frac{d}{dx}\bigl(xB_{R,P}(x^r)\bigr)=Q_{R,P}(x^r).
	\]
	Since \(xB_{R,P}(x^r)\) vanishes at \(x=0\), integration yields the identity
	\[
	\int_0^x Q_{R,P}(t^r)\,dt
	=
	\prod_{u\in R}(x^r-x_u^r)
	\int_0^x P(t^r)\,dt.
	\]
	
	Apply this identity with \(R=T\setminus T_D\). 
	For every \(i\in D\),
	\[
		\prod_{u\in T_D}(x_i^r-x_u^r)
		\int_0^{x_i}Q_{T\setminus T_D,P}(t^r)\,dt
		=
		\prod_{u\in T}(x_i^r-x_u^r)
		\int_0^{x_i}P(t^r)\,dt.
	\]
	Therefore
	$
	\theta^{T,P}_{D}
	=
	\theta^{T_D,Q_{T\setminus T_D,P}}_{D}
	\in\mathscr G(G').
	$
	
	It remains to treat the unique \(T\)-full component \(D_0\). Consider the
	sum over all components of \(G'\setminus T\):
	\[
	\theta
	=
	\sum_{D\in\mathscr C_{G'}(T)}\theta^{T,P}_{D}.
	\]
	Using the preceding identity with \(R=T\), we obtain
	\[
	\begin{aligned}
		\theta
		&=
		\sum_{i\notin T}
		\prod_{u\in T}(x_i^r-x_u^r)
		\int_0^{x_i}P(t^r)\,dt\,\partial_i\\
		&=
		\sum_{i=1}^{\ell}
		\int_0^{x_i}Q_{T,P}(t^r)\,dt\,\partial_i.
	\end{aligned}
	\]
	Write \(Q_{T,P}(y)=\sum_{k=0}^{m}q_k y^k\), with \(q_k\in S\). By the
	definition of \(\eta_k\),
	\[
	\theta=\sum_{k=0}^{m}q_k\eta_k.
	\]
	Theorem \ref{thm:basis for complete graph} and Remark \ref{rem:reduction} imply that every \(\eta_k\), \(k\geq\ell\), is
	an \(S\)-linear combination of
	\(\eta_0,\eta_1,\ldots,\eta_{\ell-1}\). Hence
	$
	\theta \in\langle\mathscr G(G')\rangle.
	$
	All terms in the definition of $\theta $, except possibly the term
	corresponding to \(D_0\), have already been shown to belong to
	\(\langle\mathscr G(G')\rangle\). 
	Therefore
	\[
	\theta^{T,P}_{D_0}
	=
	\theta
	-
	\sum_{\substack{D\in\mathscr C_{G'}(T)\\D\neq D_0}}
	\theta^{T,P}_{D}
	\in\langle\mathscr G(G')\rangle.
	\]
	
	We have proved that \(\theta^{T,P}_{D}\) belongs to
	\(\langle\mathscr G(G')\rangle\) for every component
	\(D\in\mathscr C_{G'}(T)\). The component decomposition above now gives
	\(\theta^{T,P}_{C}\in\langle\mathscr G(G')\rangle\) for every
	\(C\in\mathscr C_G(T)\). This proves
	$
	\mathscr G(G)\subseteq\langle\mathscr G(G')\rangle .
	$
\end{proof}

%
%
%

\begin{theorem}
	\label{thm:generators}
	For a simple graph $G$,
	the derivation module $\D(\A_G)$ of the arrangement $\A_G$ of type $G(r, p ,\ell)$, where $p<r$, is generated by $\mathscr{G}(G)$.
\end{theorem}

\begin{proof}
	We use induction by reducing the edges of a connected graph.
	By Theorem \ref{thm:basis for complete graph}, 
	$\D(\A_{K_\ell})$ is generated by $\eta_0, \eta_1, \dots, \eta_{\ell-1}$.
	Assume that $\D(\A_G)$ is generated by $\mathscr{G}(G)$. 
	It is enough to prove $\D(\A_{G'})$ is generated by $\mathscr{G}(G')$, 
	where $G'= (V(G), E(G)\backslash \{v,w\})$ for $ \{v, w\} \in E(G)$.

	We have $m$ distinct paths $B_j =\{v_1^j, v_2^j, \dots, v_{n_j}^j\}  $ $(j=1,2,\dots,m) $ between $v, w$.
	Suppose $\theta \in \D(\A_{G'})$,
	we want to construct a derivation $\eta \in \D(\A_G)$ 
	which is the sum of $\theta$ and elements in $\mathscr{G}(G')$.
	For $\zeta$ in the cyclic group $C_r$,
	
	Let
	$
	F=x_v^r-x_w^r=\prod_{\zeta^r=1}(x_v-\zeta x_w).
	$
	By the Leibniz rule,
	\[
	\theta(F)
	=\frac{F}{x_v-\zeta x_w}\theta(x_v-\zeta x_w)
	+(x_v-\zeta x_w)\theta\left(\frac{F}{x_v-\zeta x_w}\right).
	\]
	Moreover,
	\[
	\frac{F}{x_v-\zeta x_w}
	=\sum_{i=0}^{r-1}x_v^{r-1-i}(\zeta x_w)^i
	\equiv r x_v^{r-1}\pmod{x_v-\zeta x_w}.
	\]
	Then
	\begin{align}
		\label{equ:theta_F}
		\theta(F)
		\equiv r x_v^{r-1}\theta(x_v-\zeta x_w)
		\pmod{x_v-\zeta x_w}.
	\end{align}
	 
	Since hyperplanes $x_i=0$ belong to $\A_{G'}$,
	we have
	$
	\theta(x_i)=x_i g_i$
	for some
	$g_i\in S.
	$
	Set \(h=g_v-g_w\). Then
	\[
	\theta(F)
	=rx_v^{r}g_v-rx_w^{r}g_w
	=r x_v^r h+rF g_w.
	\]
	
	For each path \(B_j\), the same argument as in equation \eqref{equ:xvr-xwr}
	gives
	\[
	\theta(F)\in
	J_j:=
	\langle F,\,x_v^r-x_i^r\mid i\in B_j\rangle
	=
	\langle x_v^r-x_i^r\mid i\in B_j\cup\{w\}\rangle.
	\]
	Since $rFg_w \in J_j$,
	it follows that \(x_v^r h\in J_j\). 
	Let
	$Z_{ \mathrm{new} }$ denote the set associated with
	$B_j'=B_j\cup\{w\}$ defined before Lemma \ref{lem:z_new}, 
	that is,
	\begin{equation*}
		Z_{ \mathrm{new} }=\bigl\{ \prod_{u\in U} (x_v^r-x_u^r) \mid U \subseteq \{1,2,\dots,\ell\}, \ 
		U \cap B_j' \neq \emptyset \ (1 \leq j \leq m) \bigr\},
	\end{equation*} 
	where the condition on $U$ is as stated in the set definition.
	Applying Lemma \ref{lem:z_new} directly gives
	\(J=\bigcap_{j=1}^{m}J_j=\langle Z_{ \mathrm{new} }\rangle\).
	Since \(x_v^r h\in J\), Lemma \ref{lem:xvf} gives
	\(h\in J=\langle Z_{ \mathrm{new} }\rangle\).
	If \(w\in U\), the corresponding generator of \(Z_{ \mathrm{new} }\)
	contains the factor \(F=x_v^r-x_w^r\). 
	If \(w\notin U\), then \(U\) meets
	every \(B_j\). 
	Hence
	\[
	h=FF_0+\sum_{U\in M}
	\left(\prod_{u\in U}(x_v^r-x_u^r)\right)F_U,
	\]
	where \(F_0,F_U\in S\) and
	$
	M=
	\{U\subseteq V(G')\setminus\{v,w\}\mid
	U\cap B_j\neq\varnothing\text{ for every }j\}.
	$
	
	For each \(U\in  M\), choose a minimal \((v,w)\)-separator
	\(T_U\subseteq U\), and let \(C_U\) be the component of
	\(G'\setminus T_U\) containing \(v\). Put \(R_U=U\setminus T_U\) and
	$
	A_{R_U}(y)=\prod_{u\in R_U}(y-x_u^r).
	$
	Define $P_U(y)$ by
	$
	P_U(y)=A_{R_U}(y)+ryA_{R_U}'(y)\in S[y].
	$
	Since
	$
	\frac{d}{dx}\bigl(xA_{R_U}(x^r)\bigr)=P_U(x^r),
	$
	we obtain
	$
	\int_0^xP_U(t^r)\,dt
	=xA_{R_U}(x^r)
	=x\prod_{u\in R_U}(x^r-x_u^r).
	$
	The arbitrary polynomials \(F_U\) remain \(S\)-module coefficients and
	are not absorbed into \(P_U\).
	
	Let \(\delta_U=\theta^{T_U,P_U}_{C_U} \in\langle\mathscr G(G')\rangle\). Since
	\(v\in C_U\) and \(w\notin C_U\),
	$
	\delta_U(x_v)
	=x_v\prod_{u\in U}(x_v^r-x_u^r),
	$
	$
	\delta_U(x_w)=0,
	$
	and consequently
	$
	\delta_U(F)
	=r x_v^r\prod_{u\in U}(x_v^r-x_u^r).
	$
	Set
	$
	\delta=\sum_{U\in M}F_U\delta_U,
	\eta=\theta-\delta.
	$
	We have \(\delta\in\langle\mathscr G(G')\rangle\), 
	$\eta \in \D(\A_{G'})$
	, and the expression for
	\(h\) gives
	\begin{align*}
		\eta(F)
		&=\theta(F)-\delta(F) \\
		&=rx_v^rh+rFg_w-\sum_{U\in M}F_U\delta_U(F) \\
		&=rF(x_v^rF_0+g_w)\in FS.
	\end{align*}
	Since \(F\mid\eta (F)\),
	we have $\eta (F)\equiv 0$ (mod $x_v-\zeta x_w$) for every $\zeta^r=1$.
	On the other hand,
	equation \eqref{equ:theta_F} applied to $\eta $ gives 
	$
	0 \equiv
	\eta (F)
	\equiv r x_v^{r-1}\eta (x_v-~\zeta x_w)
	\pmod{x_v-\zeta x_w}.
	$
	Because 
	$x_v- \zeta x_w$ is prime and 
	$x_v- \zeta x_w$ does not divide $r x_v^{r-1}$,
	it follows that
	$
	\eta (x_v-\zeta x_w)
	\in(x_v-\zeta x_w)S
	$
	for every
	$
	\zeta^r=1.
	$
	Thus \(\eta \in D(\mathcal A_G)\). 
	
	By the induction hypothesis and
	Lemma \ref{lem:generators},
	$
	\eta
	\in\langle\mathscr G(G)\rangle
	\subseteq\langle\mathscr G(G')\rangle.
	$
	Finally,
	$
	\theta=\eta +\delta
	\in\langle\mathscr G(G')\rangle.
	$
	This completes the induction.
\end{proof}

The following descending chain,
defined by M\"uhlherr \cite{l}, 
plays a significant role in reducing the infinite set $\mathscr{G}(G)$ to a finite set.

\begin{definition}
	\label{def:descending chain}
	Let $T$ be a minimal separator of $G$, $C=\{v_1, v_2, \dots, v_n\}$ be a connected component of $G_{V\backslash T}$.
	Set $C_i=\{v_1, v_2, \dots, v_i\}$, and $T_i=\bigl(\cup_{j=1,2,\dots,i} N(v_j) \bigr)\backslash C_i $ for $i=1,2,\dots,n$. 
	The pairs $(T_i, C_i)$ $(1\le i \le n)$ are called a descending chain of $(T, C)$.
\end{definition}

\begin{remark}
	\label{rem:descending chain}
	We have $T_n = T$ and $C_n=C$.
	By definition, $T_i$ may not be a minimal separator, 
	and $C_i$ may be a union of connected components. 
	We obtain a different descending chain if the chosen order of vertices of $C$ is different.
\end{remark}

\begin{lemma}
	\label{lem:descending chain}
	Given a minimal separator $T$, and $C\in \mathscr{C}(T)$. 
	The derivation $\theta_C^{T,P}$ can be generated by the derivations $\theta_{C_i}^{T_i,1}$ defined by descending chain of $(T, C)$  
	for $1\le i \le n=|C|$.
\end{lemma}

\begin{proof}
	It is enough to show $\theta_C^{T, y^m}$ is generated by $\theta_{C_i}^{T_i,1}$ for $m \ge 0$.
	We use induction on $i$ to prove $\theta_{C_i}^{T_i, y^m} \in \langle \theta_{C_j}^{T_j, P_j}\ (j < i), \theta_{C_i}^{T_i, 1} \rangle$,
	where $P_j\in S[y]$.
	When $i=1$, $C_1$ has a single vertex $v_1$,
	we have
	\begin{align*}
		\theta_{C_1}^{T_{1}, y^m} & = \prod_{t\in T_{1}} (x_{v_1}^r-x_t^r) \int_0^{x_{v_1}} t^{mr} dt \partial_{v_1}\\
								  & \doteq x_{v_1}^{mr} \cdot \prod_{t\in T_{1}} (x_{v_1}^r-x_t^r) \int_0^{x_{v_1}} 1 dt \partial_{v_1} \\
								  & = x_{v_1}^{mr} \cdot \theta_{C_1}^{T_{1}, 1}.
	\end{align*}
	Note that $T_{i-1} \subseteq T_i \cup \{v_i\}$. 
	For $i > 1$,  
	we have 
	\begin{align*}
		(mr+1)\theta_{C_i}^{T_{i}, y^m} - x_{v_i}^{mr} \cdot \theta_{C_i}^{T_{i}, 1} 
			& = \sum_{j\in C_i} \prod_{t\in T_i} (x_j^r-x_t^r) \cdot (x_{j}^{mr}- x_{v_i}^{mr})\cdot x_j \partial_j    \\
			& = \sum_{j\in C_i} \prod_{t\in T_i} (x_j^r-x_t^r) \cdot (x_{j}^{r}- x_{v_i}^{r}) \int_0^{x_{j}} H(t^r) dt \partial_j    \\
			& = \sum_{j\in C_{i-1}} \prod_{t\in T_i} (x_j^r-x_t^r) \cdot (x_{j}^{r}- x_{v_i}^{r}) \int_0^{x_{j}} H(t^r) dt \partial_j \\
			& = \sum_{j\in C_{i-1}} \prod_{t\in T_{i-1}} (x_j^r-x_t^r) \cdot \prod_{t\in T_{i} \cup \{v_i\} \backslash T_{i-1} } (x_{j}^{r}- x_t^{r}) \int_0^{x_{j}}  H(t^r) dt \partial_j \\
			& = \sum_{j\in C_{i-1}} \prod_{t\in T_{i-1}} (x_j^r-x_t^r)  \int_0^{x_{j}} J(t^r) dt \partial_j \\
			& = \theta_{C_{i-1}}^{T_{i-1}, J}
	\end{align*}
	where $H, J\in S[y]$.
	By induction, the proof is complete.
\end{proof}

\begin{theorem}
	\label{thm:reduce theta_k}
	Given a non-complete graph $G=(V, E)$ with connectivity $\kappa(G)$.
	The vector fields $\eta_k$ $(k\ge \kappa(G))$ is generated by 
	$\eta_0, \eta_1,\dots, \eta_{ \kappa(G) -1 }$ and all derivations $\theta_C^{T, 1}$, 
	where $T$ is a separator of $G$ and $C\in \mathscr{C}(T)$.
\end{theorem}

\begin{proof} 
	For a separator $T$ with $|T| = k \ge \kappa(G)$,
	consider 
	\begin{equation*}
		\theta 
		      = \sum_{C\in \mathscr{C}(T)} \theta_C^{T,1}
		      = \sum_{i\in V \backslash T} \prod_{t\in T} (x_i^r-x_t^r)\cdot x_i \partial_i 
		      = \theta_{V\backslash T}^{T,1}.
	\end{equation*}
	The entry $\prod_{t\in T} (x_i^r-x_t^r)\cdot x_i$ can be expressed by $f_0 x_i +f_1 x_i^{r+1}+\dots+ x_i^{kr+1}$, 
	for $f_i \in S$. 
	Then we have $\theta_{V\backslash T}^{T,1} - \sum_{i=0,1,\dots, k-1} f_i \eta_i = \eta_k$. 
	Since $\theta_{V\backslash T}^{T,1}$ is a sum of derivations $\theta_C^{T, 1}$ over $C \in \mathscr{C}(T)$,
	$\eta_k$ is generated by $\eta_0, \eta_1, \dots, \eta_{k-1}$ and $\theta_C^{T,1}$.
	By induction on $k$, the proof is complete.
\end{proof}

Then we can obtain the main theorem \ref{thm:main}. 

\begin{proof}{(Proof of Theorem \ref{thm:main})}
	From Theorem \ref{thm:generators}, Lemma \ref{lem:descending chain}, and Theorem \ref{thm:reduce theta_k}, the proof of main theorem is complete.
\end{proof}

Note that the derivation $\eta_{ \kappa(G) }$ can be excluded in the main theorem when $G$ is not complete.
And there are still too many generators of the form $\theta_C^{T,1}$.

\begin{corollary}
	\label{cor:reduce theta_k}
	Let G be a non-complete graph.
	The vector fields $\eta_k$ $(k\ge \kappa(G))$ is generated by 
	$\eta_0, \eta_1,\dots, \eta_{ \kappa(G) -1 }$ 
	and derivations $\theta_{C_i}^{T_i,1}$ defined by a descending chain $(T_i,C_i)$ of $(T, C)$, 
	for all minimal separators $T$ and $C\in \mathscr{C}(T)$.
\end{corollary}

\begin{proof} 
	The above theorem has proved the case $|T|=\kappa(G)$.
	For a separator $T$ with $|T| = k \ge \kappa(G)$,
	there is a minimal separator $T' \subseteq T$.
	Consider
	\begin{align*}
		\theta 
		= \sum_{C\in \mathscr{C}(T)} \theta_C^{T,1}
		& = \sum_{i\in V \backslash T} \prod_{t\in T} (x_i^r-x_t^r)\cdot x_i \partial_i \\
		& = \sum_{i\in V \backslash T'} \prod_{t\in T'} (x_i^r-x_t^r) \cdot \prod_{t\in T \backslash T'} (x_i^r-x_t^r) \cdot x_i \partial_i  \\
		& = \sum_{i\in V \backslash T'} \prod_{t\in T'} (x_i^r-x_t^r) \cdot  \int_0^{x_i} P(t^r) dt \partial_i  \\
		& = \theta_{V\backslash T'}^{T', P}.
	\end{align*}
	The entry $\prod_{t\in T} (x_i^r-x_t^r)\cdot x_i$ in $\theta_{V\backslash T'}^{T', P}$ 
	can be expressed by $f_0 x_i +f_1 x_i^{r+1}+\dots+ x_i^{kr+1}$.
	Similarly, we have
	$\theta_{V\backslash T'}^{T', P} - \sum_{i=0,1,\dots, k-1} f_i \eta_i = \eta_k$.
	Since $\theta_{V\backslash T'}^{T', P}$ is generated by 
	the derivations $\theta_{C_i}^{T_i, 1}$ 
	defined by descending chain $(T_i, C_i)$ of $(T', C')$ $(C'\in \mathscr{C}(T'))$,
	the proof is complete by induction on $k$.
\end{proof}

Based on Theorem  \ref{thm:generators}, Lemma \ref{lem:descending chain}, and Corollary \ref{cor:reduce theta_k} above, we obtain the following theorem.

\begin{theorem}
	\label{thm:main reduce}
	The vector fields $\eta_0, \eta_1, \dots, \eta_{\kappa(G)}$ 
	and derivations $\theta_{C_i}^{T_i,1}$ defined by a descending chain $(T_i,C_i)$ of $(T, C)$ 
	(for all minimal separators $T$ and $C\in \mathscr{C}(T)$) 
	form a generating set for derivation module $\D(\A_G)$ of arrangement $\A_G$ of type $G(r, p, \ell)$ where $p<r$. 
\end{theorem}

\begin{example}[Star graph]
	Let $m\geq2$, and let $G$ be a star graph on
	$V=\{1,2,\ldots,\ell \}$ with edge set
	$
	E=\bigl\{\{1,i\}\mid 2\leq i\leq\ell \bigr\}.
	$
	Consider the arrangement $\A_G$ of type $G(r,p,\ell)$, where $p<r$.
	Its defining polynomial is
	$
	Q(\A_G)
	=\left(\prod_{j=1}^{\ell}x_j\right)
	\left(\prod_{i=2}^{\ell}(x_i^r-x_1^r)\right).
	$
	The set $T=\{1\}$ is the unique minimal separator of $G$, and
	$
	\mathscr C_G(T)=\bigl\{\{2\},\{3\},\ldots,\{\ell \}\bigr\}.
	$
	For each $2\leq i\leq \ell$, the derivation corresponding to the component
	$\{i\}$ is
	$
	\theta^{\{1\},1}_{\{i\}}
	=x_i(x_i^r-x_1^r)\partial_i.
	$
	By Saito's criterion (Theorem 4.19, \cite{ot}),
	these derivations,
	together with
	$
	\eta_0=\sum_{i=1}^{\ell}x_i\partial_i,
	$
	form a basis of
	$\D(\A_G)$. 
	In this special case, the derivations
	$\eta_i$ $(i \ge 0)$ 
	and $\theta_C^{T, 1}$ (all minimal separators $T$) 
	are sufficient to generate $\D(\A_G)$.
\end{example}

The following two examples demonstrate that,
for general arrangements $\A_G$, 
the derivations
$\eta_i$ $(0 \le i \le \kappa(G))$ 
and $\theta_C^{T, 1}$ (all minimal separators $T$) 
are insufficient to generate $\D(\A_G)$
in Theorem \ref{thm:main}.

\begin{example}
	We use the same graph $G$ from Example 3.9 of \cite{l}.
	The arrangement $\A_G$ corresponding to the graph $G$ in Fig.\ref{fig:four} is defined by
	$x_1 x_2 x_3 x_4 (x_1^r - x_4^r ) (x_2^r - x_3^r ) (x_2^r - x_4^r ) (x_3^r - x_4^r )$.
	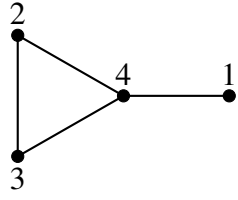
\begin{figure}[htbp]
		\centering
		\begin{tikzpicture}[scale=0.8]
			\draw[thick, black] (0,0) -- (0,2)  ; 
			\draw[thick, black] (0,2) -- (1.75,1)  ;
			\draw[thick, black] (0,0) -- (1.75,1)  ;
			\draw[thick, black] (1.75,1)  -- (3.5,1) ; 
			\node[above] at (3.5,1) {1};
			\node[above] at (0,2) {2};
			\node[below] at (0,0) {3};
			\node[above] at (1.75,1) {4};
			\filldraw[black] (0,0) circle (0.1);
			\filldraw[black] (0,2) circle (0.1);
			\filldraw[black] (1.75,1) circle (0.1);
			\filldraw[black] (3.5,1) circle (0.1);
		\end{tikzpicture}
		\caption{A graph on 4 vertices} 
		\label{fig:four}
	\end{figure}
	
	The derivations $\theta_{\{1\} }^{ \{4\} , 1 }, \theta_{\{2 , 3\} }^{ \{4\} , 1 }$ are given by the unique minimal separator $\{4\}$,
	and the derivation $\theta_{\{3\} }^{ \{2, 4\} , 1 }$ comes from a non-minimal separator $\{2, 4\}$.
	It is easy to check 
	$\D(\A_G)=\langle \eta_0, \theta_{\{1\} }^{ \{4\} , 1 }, \theta_{\{2 , 3\} }^{ \{4\} , 1 }, \theta_{\{3\} }^{ \{2, 4\} , 1 } \rangle$ by Saito's criterion (Theorem 4.19, \cite{ot}).
	Thus $\A_G$ is free.
	Note that $\eta_1 = x_4^r \eta_0 + \theta_{\{1\} }^{ \{4\} , 1 } + \theta_{\{2 , 3\} }^{ \{4\} , 1 }$.
	Clearly, $\D(\A_G)$ can not be generated by derivations $\eta_0, \eta_1, \theta_{\{1\} }^{ \{4\} , 1 }, \theta_{\{2 , 3\} }^{ \{4\} , 1 }$. 
\end{example}

\begin{example}
	 The arrangement
	 $\A_G$ corresponding to the graph $G$ in Fig.\ref{fig:five} is defined by
	 $
	 x_1x_2x_3x_4x_5
	 (x_1^r-x_2^r)(x_1^r-x_3^r)(x_2^r-x_3^r)
	 (x_1^r-x_4^r)(x_1^r-x_5^r)
	 (x_2^r-x_4^r)(x_2^r-x_5^r)(x_4^r-x_5^r).
	 $
	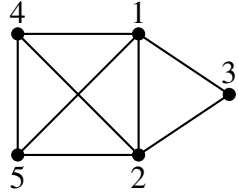
\begin{figure}[htbp]
		\centering
		\begin{tikzpicture}[scale=0.8]
			\draw[thick, black] (0,0) -- (0,2)  ; 
			\draw[thick, black] (0,0) -- (2,0)  ;
			\draw[thick, black] (2,0) -- (2,2)  ;
			\draw[thick, black] (0,2) -- (2,2)  ;
			\draw[thick, black] (0,0) -- (2,2)  ;
			\draw[thick, black] (2,0) -- (0,2)  ;
			\draw[thick, black] (3.5,1) -- (2,2)  ;
			\draw[thick, black] (3.5,1) -- (2,0)  ;
			\node[above] at (3.5,1) {3};
			\node[above] at (0,2) {4};
			\node[below] at (0,0) {5};
			\node[below] at (2,0) {2};
			\node[above] at (2,2) {1};
			\filldraw[black] (0,0) circle (0.1);
			\filldraw[black] (0,2) circle (0.1);
			\filldraw[black] (2,0) circle (0.1);
			\filldraw[black] (2,2) circle (0.1);
			\filldraw[black] (3.5,1) circle (0.1);
		\end{tikzpicture}
		\caption{A graph on 5 vertices} 
		\label{fig:five}
	\end{figure}
	
	The unique minimal separator $\{1,2\}$ gives exactly the two derivations
	$
	\theta^{\{1,2\},1}_{\{3\}},
	\theta^{\{1,2\},1}_{\{4,5\}}.
	$
	Thus the proposed generating module is
	$
	M=\langle\eta_0,\eta_1,\eta_2,\theta^{\{1,2\},1}_{\{3\}},\theta^{\{1,2\},1}_{\{4,5\}}\rangle.
	$
	Now $\{1,2,4\}$ is a non-minimal separator.
	Lemma \ref{lemma:derivation} gives the derivation 
	$\theta^{\{1,2,4\},1}_{\{5\}}\in \D(\A_G)$. 
	We show that
	$\theta^{\{1,2,4\},1}_{\{5\}} \notin M $.
	Suppose that
	\[
	\theta^{\{1,2,4\},1}_{\{5\}}=f_0\eta_0+f_1\eta_1+f_2\eta_2+g_1\theta^{\{1,2\},1}_{\{3\}}+g_2\theta^{\{1,2\},1}_{\{4,5\}},
	\]
	for some $f_0, f_1, f_2, g_1, g_2 \in S$.
	Put $z_i=x_i^r$ and $P(y)=f_0+f_1y+f_2y^2$. Comparing the coefficients
	at $\partial_1$ and $\partial_2$ gives
	$
	P(z_1)=P(z_2)=0,
	$
	and
	$
	P(y)=h(y-z_1)(y-z_2)
	$
	for some $h\in S$. 
	 The coefficient of $\partial_4$ in $\theta^{\{1,2,4\},1}_{\{5\}}$ is zero, whereas its
	 coefficient on the assumed right-hand side is
	 $
	 x_4(z_4-z_1)(z_4-z_2)(h+g_2).
	 $
	 Therefore $h+g_2=0$. It follows that the coefficient of $\partial_5$ on
	 the right-hand side is also zero:
	 $
	 x_5(z_5-z_1)(z_5-z_2)(h+g_2)=0.
	 $
	 On the other hand, the coefficient of $\partial_5$ in $\theta^{\{1,2,4\},1}_{\{5\}}$ is non-zero.
	Thus
	$
		\theta^{\{1,2,4\},1}_{\{5\}}
		\notin
		\left\langle
		\eta_0,\eta_1,\eta_2,
		\theta^{\{1,2\},1}_{\{3\}},
		\theta^{\{1,2\},1}_{\{4,5\}}
		\right\rangle.
	$
	Hence the derivations $\eta_i$, $0\leq i\leq\kappa(G)$, together with
	all derivations associated with minimal separators do not generate $\D(\A_G)$.
\end{example}

\section{Discussion}\label{sec:conjecture}

Given a simple graph $G$,
for $v\in V(G) $, 
define a vector field $\theta_v^{N(v)}$ by $\prod_{t\in N(v)} (x_v^r-x_t^r) \partial_v$.
It is easy to check $\theta_v^{N(v)} (x_i-\zeta x_j) \in (x_i-\zeta x_j)S $ for $\{i, j\} \in E(G)$. The authors \cite{FWY} have obtained a basis for  $\A_{K_\ell}$ of type $G(r, r, \ell)$ shown in the following.

\begin{theorem}
	\label{thm:conjecture} 
	$($$\mathrm{Theorem\ 3.4}$, \cite{FWY}$)$ 
	The vector fields $x_1^{r-1} x_2^{r-1} \dots x_\ell^{r-1} \cdot \eta_{-1}, \eta_0, \eta_1, \dots, \eta_{\ell-2}$ 
	form a basis for $\D(\A_{K_\ell})$ of type $G(r, r, \ell)$. 
\end{theorem}

Note that $x_1^{r-1} x_2^{r-1} \dots x_\ell^{r-1}\cdot\eta_{-1} \notin \langle \eta_0, \eta_1, \dots, \eta_{\ell-1}  \rangle$.
By observation of derivations of $\A_G$ of type $G(r, r ,\ell)$,
we pose the following conjecture.

\begin{conjecture}
	Given a simple connected graph $G$.
	The vector fields $x_1^{r-1} x_2^{r-1} \dots x_\ell^{r-1}~\cdot~\eta_{-1}, \eta_0, \eta_1, \dots, \eta_{\kappa(G)}$,
	$\theta_v^{N(v)}$
	(for $v\in V(G) $), 
	and a finite set of vector fields $\theta_{C}^{T,1}$ 
	(for separators $T$ and $C\in \mathscr{C}(T)$) 
	form a generating set for derivation module $\D(\A_G)$ of arrangement $\A_G$ of type $G(r, r, \ell)$. 
\end{conjecture}



\section*{Acknowledgments}
The authors thank Ye Liu for helpful discussions and the reviewers for their valuable comments and suggestions.
This research was funded by National Natural Science Foundation of China (No. 12201029).

\section*{Declarations}

\noindent \textbf{Competing Interests} The authors have no competing interests to declare.

\end{document}